\documentclass[10pt]{article}

\usepackage{color}
\usepackage{amssymb}
\usepackage{amsthm}
\usepackage{amsmath}
\usepackage{makeidx}
\usepackage{amsfonts}
\usepackage{enumerate}
\usepackage{graphicx}
\usepackage{float}
\usepackage[margin=2.5cm]{geometry}
\usepackage{tikz}
\usepackage{hyperref}

\definecolor{blau}{rgb}{0.1,0.0,0.9}
\definecolor{gruen}{cmyk}{1.0,0.2,0.7,0.07}
\definecolor{mag}{cmyk}{0.0,0.9,0.3,0.0}

\usepackage{enumerate}
\theoremstyle{plain}
\newtheorem{theorem}{Theorem}
\newtheorem{lemma}[theorem]{Lemma}
\newtheorem{corollary}[theorem]{Corollary}
\newtheorem{problem}[theorem]{Problem}

\theoremstyle{definition}
\newtheorem{defn}{Definition}

\newcommand{\T}{^\text{T}}
\newcommand{\per}{\mathrm{per}}
\begin{document}

\title{Restricted completion of sparse partial Latin squares}

\author{Lina J. Andr\'en \and Carl Johan Casselgren \and Klas Markstr\"om}

\maketitle

\begin{abstract}
An $n \times n$
partial Latin square $P$ is called $\alpha$-dense if each row and column
has at most $\alpha n$ non-empty cells and each symbol occurs at most
$\alpha n$ times in $P$. An $n \times n$ array $A$ where each cell
contains a subset of $\{1,\dots, n\}$ is a 
$(\beta n, \beta n, \beta n)$-array if each symbol occurs at most
$\beta n$ times in each row and column and 
each cell contains a set of size
at most $\beta n$. Combining the notions of 
completing partial Latin squares
and avoiding arrays, we prove that 
there are constants $\alpha, \beta > 0$ such that, for every
positive integer $n$,
if $P$ is an $\alpha$-dense $n \times n$ partial Latin square,
$A$ is an $n \times n$ $(\beta n, \beta n, \beta n)$-array, and no cell of
$P$ contains a symbol that appears in the corresponding cell of $A$, then 
there is a completion of $P$ that avoids $A$; that is,
there is a Latin square $L$ that agrees with $P$ on every non-empty 
cell of $P$, and, for each $i,j$ satisfying $1 \leq i,j \leq n$, the symbol
in position $(i,j)$ in $L$ does not appear 
in the corresponding cell of $A$.
\end{abstract}

\section{Introduction}

Consider an $n\times n$ array $A$
where each cell contains a
subset of the symbols in $[n]=\{1,\dots,n\}$.
If no cell in $A$ contains a set of size larger than 
$m_1$, and if no symbol occurs more than 
$m_2$ times in any row or more than 
$m_3$ times in any column, then $A$ is 
an \emph{ $(m_1,m_2,m_3)$-array
(of order $n$)}. A $(1,1,1)$-array is usually called 
a \emph{partial Latin square} (or PLS), and such an 
array with no empty cell is a {\em Latin square}.
The cell in position $(i,j)$ of $A$ is denoted by
$(i,j)_A$, and the set of symbols in cell $(i,j)_A$
is denoted by $A(i,j)$. By slight abuse of notation,
if $L$ is a (partial) Latin square, then $L(i,j)$ usually
denotes the symbol in cell $(i,j)_L$, that is,
$L(i,j) =k$. Moreover, the symbol $L(i,j)$ is called an 
{\em entry} of cell $(i,j)_L$.

An $n \times n$ 
partial Latin square $P$ is called {\em $\alpha$-dense}
if each row and column contains at most $\alpha n$
non-empty cells and each symbol appears at most $\alpha n$
times in $P$.
An $n \times n$ partial Latin square $P$ is {\em completable}
if there is an
$n \times n$ Latin square $L$ such that
$L(i,j) = P(i,j)$ for each non-empty cell $(i,j)_P$
of $P$; $L$ is also called a {\em completion} of $P$. 
Similarly, an $n \times n$ 
array $A$ is {\em avoidable} if there is an 
$n \times n$
Latin square $L$ such that for each $i,j$
satisfying $1 \leq i,j \leq n$,
$L(i,j)$ does not appear in
cell $(i,j)_A$ of $A$; we also say that $L$
{\em avoids} $A$.

The problem of completing partial Latin squares
	is a classic within combinatorics and there
	is a wealth of results in the literature. Let us here
	just mention a few classic and recent results.
	In general, it is an $NP$-complete problem to determine
	if a partial Latin square is completable \cite{Colbourn}.
	Thus it is natural to ask if particular families of
	partial Latin squares are completable. 
	A classic result due to Ryser \cite{Ryser} 
	states that if $n \geq r,s$, then
	every $n \times n$ partial Latin square whose non-empty cells lie in
	an $r \times s$ Latin rectangle $Q$ is completable if and only if each
	of the symbols $1,\dots,n$ occurs at least $r+s-n$ times in $Q$.
	Another classic result
	is Smetaniuk's proof \cite{Smetaniuk}
	of Evans' conjecture \cite{Evans} that every 
	$n \times n$ partial Latin square
	with at most $n-1$ entries is completable.
	This was also independently proved by Andersen
	and Hilton \cite{AndersenHilton}.
	Adams et al. \cite{AdamsBryantBuchanan} characterized which
	partial Latin squares with 
	$2$ filled rows and columns are completable
	and by results of Casselgren et al. \cite{CasselgrenHaggkvist}
	and Kuhl et al. \cite{KuhlSchroeder},
	all partial Latin squares of order at least $6$
	with all entries in one fixed column or row, or containing a
	prescribed symbol, is completable. 
	Building on techniques by Chetwynd and 
	H\"aggkvist \cite{ChetwyndHaggkvist} and 
	Gustavsson \cite{Gustavsson},
	Bartlett \cite{Bartlett} proved that
	every $\epsilon$-dense partial Latin square is completable,
	provided that $\epsilon < 9.8 \cdot 10^{-5}$.
	
The problem of avoiding arrays
was first posed by H\"aggkvist~\cite{Haggkvist}.
He also found the first (non-trivial) family of avoidable arrays:
If $n = 2^k$ and $P$ is a 
	$(1,n,1)$-array of order $n$ with
	empty last column,
	then $P$ is avoidable.
In his original paper \cite{Haggkvist} H\"aggkvist also
conjectured that there is constant $c >0$
such that for every positive integer $n$,
every $(cn, cn, cn)$-array is avoidable.
Andr\'en \cite{AndrenEven} established that
H\"aggkvist's conjecture holds for arrays of even order
and the case of odd order arrays was settled
by Andr\'en et al. \cite{AndrenOdd} confirming
H\"aggkvist's conjecture in the affirmative.
Related results appear in 
\cite{ChetwyndRhodes2, CutlerOhman, Casselgren};
in particular, in 
\cite{Casselgren} it is proved that is is $NP$-complete 
	to decide if 
an array with at most two symbols per cell is avoidable, 
even if only two
distinct symbols occur in the array.

	Much of the research on avoiding arrays 
	has been focused
	on avoiding arrays that contain at most one symbol in each cell,
	so-called {\em single entry arrays}.
	Most notably, 
	by results of Chetwynd and Rhodes \cite{ChetwyndRhodes1},
	Cavenagh \cite{Cavenagh} and \"Ohman \cite{Ohman},
	all partial Latin squares of order at least $4$ are avoidable.
	In \cite{ChetwyndChess}, \cite{Casselgren} 
	and \cite{MarkstromOhman}
	some families of avoidable and unavoidable 
	arrays are given.

	In this paper we combine the notions of completing partial Latin squares
	and avoiding arrays and consider
	the problem of completing a partial Latin square
	subject to the condition that the completion should
	avoid a given array as well. There are some previous
	results in this direction: 
	\"Ohman \cite{Ohman2} determined for which pairs $P, A$,
	where $P$ is a partial Latin square of order $n$
	with entries from only two distinct symbols,
	and $A$ is a single entry array of order $n$ 
	with entries only from the same
	two distinct symbols, there is a completion of $P$
	that avoids $A$. Denley et al. \cite{DenleyKuhl}
	proved that if $P$ is an $n \times n$ partial Latin square
	and $Q$ is an $n \times n$ partial Latin square that
	avoids $P$, then there is a completion of $P$ that avoids $Q$
	if $n = 4t$, $P$ contains at most $t-1$ non-empty cells
	and $t \geq 9$.
	
	Note further that the problem of determining if a given
	partial Latin square $P$ has a completion $L$ which avoids
	a given array $A$ is  certainly $NP$-complete in the general case,
	since it both contains the problem of completing partial Latin squares
	and avoiding arrays as special cases.

The main result of this paper is the following proposition
which is proved by combining techniques developed by Bartlett \cite{Bartlett}
and Andr\'en et al. \cite{AndrenOdd}.

\begin{theorem}
\label{th:main}
	There are constants
	$\alpha > 0$ and $\beta > 0$, such that for
	every positive integer $n$, if $P$ is an 
	$n \times n$ $\alpha$-dense partial Latin square,
	$A$ is an $n \times n$ $(\beta n, \beta n, \beta n)$-array,
	and no cell of $P$ contains a symbol that occurs
	in the corresponding cell of $A$, then
	there is a completion of $P$ that avoids $A$.	
\end{theorem}

In this paper we also consider random partial Latin squares and arrays.
Let $\mathcal{P}(n,p)$ denote the probability space 
of all $n \times n$ partial Latin squares $P$
where each cell $(i,j)_P$ independently is empty with 
probability $1-p$ and contains symbol
$s$ with  probability $\frac{p}{n}$, $s=1,\dots,n$, and where
we for $i=1,\dots,n$, 
empty any cell $(i,j_1)_P$ in row $i$ that contains the same entry 
as another cell $(i,j_2)_P$ in row $i$, where $j_2 > j_1$.

Using our main result we  prove the following proposition on random arrays
and random partial Latin squares.

\begin{corollary}
\label{cor:random}
	Let $P$ be a random PLS distributed as $\mathcal{P}(n,p)$, and
	let $A$ be a random $n \times n$ array
	where each cell $(i,j)_A$ of $A$ is assigned
	a set $A(i,j)$ of size $m = m(n)$ by choosing each set independently and
	uniformly at random
	from all $m$-subsets of $[n]$, and where any entry of $A$ that
	occurs in the corresponding cell of $P$ is removed. 
	There are constants $\rho_1$ and $\rho_2$ such that
	if $p< \rho_1$ and $m \leq \rho_2 n$, 
	then with probability tending to $1$,
	there is a completion of $P$ that avoids $A$.
\end{corollary}

This result is deduced from our Theorem \ref{th:main}, and
it also holds if we take $P$ to be a given (deterministic) $\alpha$-dense PLS
and $A$ a random array, or $P$ a random PLS and $A$ a given 
$(\beta n,\beta n,\beta n)$-array.

The rest of the paper is organized as follows.
In Section 2 we introduce some terminology and notation
and also outline the proof of Theorem \ref{th:main}.
Section 3 contains the proof of a slightly reformulated
version of Theorem \ref{th:main}. In Section 4 we prove
Corollary \ref{cor:random}, and in Section 5 we give some
concluding remarks; in particular, we give an example indicating
what numerical values of $\alpha$ and $\beta$ in Theorem \ref{th:main}
might be best possible. In the beginning of Section 3 we shall
present numerical values of $\alpha$ and $\beta$ for which
our main theorem holds, provided that $n$ is large enough.

\section{Terminology, notation and outline
of the proof of Thereom \ref{th:main}}

If $L$ is a Latin square, $A$ is an array, and $L$ does not avoid $A$,
then the cells $(i,j)_L$ such that $L(i,j) \in A(i, j)$
are the {\em conflict cells of $L$ with $A$} (or just the {\em conflicts}
of $L$).
If $P$ is a PLS, then the cells $(i,j)_L$
that correspond to non-empty cells in $P$ are the {\em prescribed cells
of $L$ with $P$} (or just the {\em prescribed cells}).

An \emph{intercalate} in an $n \times n$ Latin square $L$ is a set 
	$$C =\{ (r_1,c_1)_L, (r_1,c_2)_L, (r_2,c_1)_L, (r_2,c_2)_L \}$$
	of cells in $L$ such that \( L(r_1,c_1)=L(r_2,c_2)\) and 
	\(L(r_1,c_2)=L(r_2,c_1)\).
	If in addition
	$$|\{ L(r_1,c_1), L(r_1,c_2) \} \cap \{1, \dots, \lfloor n/2 \rfloor\}| = 1,$$
	then $C$ is called
	a \emph{strong intercalate}.
	
	If
	$$C = \{(r_1,c_1)_L,(r_1,c_2)_L,(r_2,c_1)_L,(r_2,c_2)_L\}$$ 
	is an intercalate in $L$ with 
	\(L(r_1,c_1)=s_1\) and \(L(r_1,c_2)=s_2\), then
	a \emph{swap on $C$}
 	is the operation \(L \mapsto L'\), where \(L'\) is a 
	Latin square with 
	$$L'(r_1,c_1)=L'(r_2,c_2)=s_2, \, L'(r_1,c_2)=L'(r_2,c_1)=s_1,$$
	and \(L'(i,j)=L(i,j)\) for all other \( (i,j)\).
	The intercalate $C$ is called 
	\emph{allowed with respect to $A$} (or just \emph{allowed})
	if performing a swap on it yields a 
	Latin square $L'$ in which none of the cells in 
	$$\{(r_1,c_1)_{L'},(r_1,c_2)_{L'},(r_2,c_1)_{L'},(r_2,c_2)_{L'}\}$$ 
	is a conflict
	cell of $L'$ with $A$.

	Let $T$ be some set of cells from a Latin square $L$.
	If there is a Latin square $L'$
	satisfying that 
	
	\begin{itemize}
	
	\item $L'(i,j) = L(i,j)$ if $(i,j)_L \notin T$, and,
	
	\item $L'(i,j) \neq L(i,j)$ for some $(i,j)_L \in T$,
	
	\end{itemize}
	then we say that $L'$ is obtained from $L$
	{\em by performing a trade on $T$}. We will also refer to the set $T$
	as a trade. Note that a swap on an intercalate
	may be seen as performing a trade on the intercalate.
	
	A \emph{generalized diagonal} $\mathcal{D}$, or simply a \emph{diagonal}, in
	an array $A$ of order $n$ is a set of $n$ cells in $A$, such that
	no two cells of $\mathcal{D}$ are in the same row or column of $A$.
	The \emph{main diagonal} in $A$ is the diagonal $\{(i,i)_A :i\in[n]\}$.
	A \emph{transversal} of a Latin square $L$ of order $n$ is a diagonal 
	$\mathcal{D}$ in $L$ such that that $\{L(r,c):(r,c)_L\in \mathcal{D}\}=[n]$.

For the proof of Theorem \ref{th:main}, we need some
previous results.
The following is due to Br\`egman \cite{Bregman}
(see also \cite{Asratian}, p. 22).

\begin{theorem}
\label{th:bregman} 
	If \(A=[A(i,j)]\) is an \(n\times n\) \( (0,1)\)-matrix with row 
	sum \(r_{i}\) 
	on the \(i\)-th row, then the permanent $\per(A)$ of \(A\) satisfies
	\begin{equation}
		\per(A) = \sum_{\sigma\in S_n}\prod_{i=1}^{n} A(i,\sigma(i))
			\le \prod_{1\le i\le n}\left( r_i! \right)^{1/r_i},
	\end{equation}
	where $S_n$ is the symmetric group of order $n$.
\end{theorem}

By a simple correspondence between $(0,1)$-matrices and 
bipartite graphs, we get the following corollary:

\begin{corollary}
\label{cor:bipbregman}
	If \(B\) is a balanced bipartite graph on $2n$ vertices  
	and $d_1,\dots,d_n$ are the degrees of the vertices in one part of $B$, 
	then the number of perfect matchings in \(B\) is at 
	most \(\prod_{1\le i\le n}\left( d_i! \right)^{1/d_i}\).
\end{corollary}

We also need some definitions on list edge coloring.
Given a graph $G$, assign to each edge $e$ of $G$ a set
$\mathcal{L}(e)$ of colors (positive integers).
Such an assignment $\mathcal{L}$ is called 
a \emph{list assignment} for $G$ and
the sets $\mathcal{L}(e)$ are referred
to as \emph{lists} or \emph{color lists}.
If all lists have equal size $k$, then $\mathcal{L}$
is called a \emph{$k$-list assignment}.
Usually, we seek a proper
edge coloring $\varphi$ of $G$,
such that $\varphi(e) \in \mathcal{L}(e)$ for all
$e \in E(G)$. If such a coloring $\varphi$ exists then
$G$ is \emph{$\mathcal{L}$-colorable} and $\varphi$
is called an \emph{$\mathcal{L}$-coloring}. 
Denote by $\chi'_L(G)$ the minimum integer $t$
such that $G$ is $\mathcal{L}$-colorable
whenever $\mathcal{L}$ is a $t$-list assignment.
We denote by $\chi'(G)$ the chromatic index of $G$,
i.e. the minimum integer $t$ such that $G$ has
a proper $t$-edge coloring.
A fundamental result in list edge coloring theory
is the following result proved by Galvin
\cite{Galvin}:

\begin{theorem}
\label{th:Galvin}
	For any bipartite multigraph, $\chi_L'(G) = \chi'(G)$.
\end{theorem}

Note further that the main result of this paper
can be formulated as a theorem on list edge coloring
of balanced complete bipartite graphs.

\bigskip

Instead of proving Theorem
\ref{th:main} we will prove the following theorem,
which is easily seen to imply Theorem \ref{th:main}.

\begin{theorem}
\label{th:main2}
	There are constants $\alpha>0$, $\beta>0$  and $n_0$,
	such that, for every positive integer $n \geq n_0$, if $P$ is an
	$\alpha$-dense partial  Latin square of order $n$, $A$ is a
	$(\beta n, \beta n, \beta n)$-array of order $n$, 
	and no entry of $P$ appears in the corresponding cell of $A$,
	then there is a completion $L$ of $P$ that avoids $A$.	
\end{theorem}

The proof of Theorem \ref{th:main2} combines techniques from \cite{AndrenOdd} and \cite{Bartlett}.
In particular, the last part of the proof is an extension of the technique developed by Bartlett
for completing $\alpha$-dense PLS.

Below we outline the proof of Theorem \ref{th:main2}.

\bigskip

\begin{enumerate}

	\item[Step I.] Find a ``starting Latin square'' $L_0$
	of order $n$,
	such that each cell in $L_0$ 
	except at most $3n+7$
	is in $\lfloor n/2 \rfloor$ strong intercalates.

	\item[Step II.] Given $A$ and $P$,
	find a pair of permutations $(\rho, \theta)$
	so that if $A'$ and $P'$ denote the arrays obtained from
	$A$ and $P$, respectively, by applying
	$\rho$ to the rows of $A$ and $P$ and $\theta$ to the columns
	of $A$ and $P$, then $P'$ and $A'$
	satisfy certain ``sparsity'' conditions
	with respect to $L_0$. These conditions will be articulated
	more precisely below.
	
	\item[Step III.] 
	Define an $n \times n$ 
	PLS
	$R$ such that a cell of $R$ is non-empty
	if and only if the corresponding cell of $L_0$
	is a conflict cell with $A'$ and the corresponding
	cell of $P'$ is empty. We shall also require that each
	symbol in $R$ is used a bounded number of times.
	Let $\hat P$ be the PLS obtained by putting $P'$ and 
	$R$ together.
	
	\item[Step IV.] Apply our modified variant
	of the technique by 
	Bartlett \cite{Bartlett} to construct from $L_0$
	a Latin square $L_q$ that is a completion
	of $\hat P$ (and thus $P'$) and which avoids $A'$.
	
	\bigskip
	
	The above construction yields
	a Latin square $L_q$ that is a completion
	of $P'$ and which avoids $A'$. However,
	in order to obtain a Latin square $S_{q}$ from $L_q$
	that is a completion of $P$ and which avoids
	$A$, we can just apply the inverses of the
	permutations $\rho$ and $\theta$
	to the rows and columns of $L_q$, respectively.
	Hence, it suffices to prove that there is a Latin square $L_q$
	as above.

\end{enumerate}

\section{Proof of Theorem \ref{th:main2}}

In the proof of Theorem \ref{th:main2} we shall
verify that it is possible to perform Steps I-IV
described in Section 2 to obtain the Latin square
$L_q$. We will not specify the value of $n_0$
in the proof, but rather assume that $n$ is large
enough whenever necessary. Since the proof
of the theorem will contain a finite number of inequalities that
are valid if $n$ is large enough, this suffices for
proving Theorem \ref{th:main2}.

The proof of Theorem \ref{th:main2} involves
	a number of other functions and parameters;
	$$\alpha, \beta, c(n), f(n), d, k, \varepsilon;$$
	and a number of inequalities that they must satisfy.
	For the reader's convenience, explicit choices for which the proof holds
	are presented here: 
	$$\alpha = \frac{1}{100000}, \quad \beta=\frac{1}{100000}, 
	\quad k=\frac{1}{500},
	\quad \varepsilon=\frac{1}{10000},$$
	$$d=\frac{1}{20}, \quad
	c(n)= \left\lfloor\frac{n}{35000}\right\rfloor, \quad
	f(n) = \left\lfloor\frac{n}{17500}\right\rfloor$$
	We remark that since the numerical values of $\alpha$ and $\beta$
	are not anywhere near what we expect to be optimal,
	we have not put an effort into choosing optimal values for
	these parameters.
	
\begin{proof}[Proof of Theorem \ref{th:main2}]
Let $P$ be an $n \times n$ $\alpha$-dense PLS
and $A$ an $n \times n$ $(\beta n, \beta n, \beta n)$-array
such that no cell of $A$ contains a symbol that occurs in the
corresponding cell of $P$.

\bigskip
\noindent
{\bf Step I:}
Below we shall define the {\em starting Latin square $L_0$}.
This Latin square was used in \cite{AndrenOdd} and \cite{Bartlett}
and also appears in the original paper by Chetwynd and H\"aggkvist \cite{ChetwyndHaggkvist} on completing
sparse partial Latin squares.

We shall give the explicit construction assuming that $n$ is even. For the case when $n$ is odd,
one can modify the construction in the even case by swapping on some intercalates and
using a transversal; the details are given in Lemma 2.1 in \cite{Bartlett}.
 
So suppose that $n=2r$.

\begin{defn}
    Let $M_{11}$ be the cyclic Latin square of order $r$
	(i.e. the Latin square corresponding to the addition table of the cyclic 
	group of order $r$). 
	Note that $M_{11}(i,j) = j-i+1$, taking $j-i+1$ modulo $r$.
	The $r \times r$ array $M_{12}$ is defined from $M_{11}$
	by setting $M_{12}(i,j) = M_{11}(i,j) + r$,
	$1 \leq i,j \leq r$.
	Let $M_{21}=M_{12}\T$ and $M_{22}=M_{11}\T$,
	where $M\T$ is the transpose of $M$, defined in the obvious way.
\end{defn}

\begin{equation*}
	M_{11}= \begin{array}{|c|c|ccc|c|}
		\hline
		1&2&3&\cdots&r-1&r\\
		\cline{1-3}\cline{5-6}
		r&1&2&\cdots&r-2&r-1\\
		\cline{1-3}\cline{5-6}
		r-1&r&1&\cdots&r-3&r-2\\
		\multicolumn{1}{|c}{\vdots}&\multicolumn{1}{c}{\vdots}&\multicolumn{1}{c}{\vdots}&\multicolumn{1}{c}{\ddots}&\multicolumn{1}{c}{\vdots}&\multicolumn{1}{c|}{\vdots}\\
		3&4&5&\cdots&1&2\\
		\cline{1-3}\cline{5-6}
		2&3&4&\cdots&r&1\\
		\hline
	\end{array}
\end{equation*}

	Now we define the $2r \times 2r$
	Latin square $M$ by letting

\begin{itemize}	
	
	\item $M_{11}$ be the
	$r \times r$ subarray in its upper left corner, 
	
	\item $M_{12}$ be the $r \times r$ subarray 
	in its upper right corner, 
	
	\item $M_{21}$ be the $r \times r$ subarray 
	in its lower left corner, and
	
	\item $M_{22}$ be the $r \times r$ subarray 
	in its lower right corner.
	
\end{itemize}

\begin{equation*}
	M=\begin{array}{|c|c|}
		\cline{1-2}
		M_{11} & M_{12}\\
		\cline{1-2}
		M_{21} & M_{22}\\
		\cline{1-2}
	\end{array}
\end{equation*}
Every cell in $M$ belongs to a large number of strong intercalates:

\begin{lemma}
	Each cell $(i,j)_M$ in $M$ belongs to exactly $r$ distinct
	strong intercalates.
\end{lemma}
      
      \begin{proof}
Without loss of generality, 
we assume that $1\le i,j\le r$. 
	It is easy to verify that for 
	every $l \in \{1,\dots, r\}$,
	$$\{(i,j)_M, (i,r+l)_M, (r+l+j-i,j)_M, (r+l+j-i,r+l)_M\}$$ 
	is a strong intercalate in $M$.
	Hence each cell $(i,j)_M$ is
	in at least $r$ strong intercalates, and
	since a strong intercalate is uniquely determined 
	by two cells, it follows from the definition of $M$
	that each cell is in at most $r$ strong intercalates.
\end{proof}

The case when $n = 2r+1$ is not as elegant;
as mentioned above, using the Latin square
$M$ one can construct a Latin square $M'$
of order $2r+1$ such that all but
at most $3n + 7$ cells are in $\lfloor n/2 \rfloor$
strong intercalates. 
In particular, there is a row and column in $M'$ where
no cell belong to at least $\lfloor n/2 \rfloor$ strong
intercalates.
The full proof appears
in \cite{Bartlett} and therefore 
we omit the details here.

We define $L_0 := M$ when $n$ is even,
and $L_0 := M'$ when $n$ is odd.

\bigskip
\bigskip
\noindent
{\bf Step II:}
Let $A'$ be an $n \times n$ 
$(\beta n, \beta n, \beta n$)-array,
$P'$ an $n \times n$ $\alpha$-dense PLS
and $L$ a Latin square.
If the following conditions hold,
then $L$ is \emph{well-behaved} 
with respect to $A'$ and $P'$ (or just
{\em well-behaved} when $A'$ and $P'$ are clear from the context):

\begin{enumerate}[(a)]

  \item\label{condition:manyintercalates} all cells in $L$,
	except for $3n + 7$, belong to at least 
	$\lfloor n/2 \rfloor - \varepsilon n$ allowed strong intercalates;
  
  \item\label{condition:rowconflicts} each 
	row of $L$ contains at most $c(n)$ conflicts with $A'$;
  
  \item\label{condition:columnconflicts} each column of $L$ 
	contains at most $c(n)$ conflicts with $A'$;

  \item\label{condition:symbolconflicts} for each 
	symbol \(s\in[n]\) there are at most \(c(n)\) cells in 
	$L$ that contain \(s\) and that are conflicts with $A'$;
  
  \item\label{condition:symbolprescriptions} for each 
	symbol \(s\in[n]\) there are at most \(c(n)\) cells in 
	$L$ that contain \(s\) and satisfy that the corresponding cell in 
	$P'$ is non-empty;
  
  \item\label{condition:symboldiagonal} for each pair of symbols 
	\(s_1,s_2\in[n]\) there are at
	most \(c(n)\) cells in $L$ with entry $s_1$ 
	such that  $s_2$ belongs 
	to the corresponding cell in $A'$.

\end{enumerate}

We shall prove that there is a pair of permutations
$(\rho, \theta)$ such that if $\rho$ is applied
to the rows of the given arrays $A$ and $P$,
and $\theta$ is applied to the columns of $A$ and $P$,
then the resulting arrays $A'$ and $P'$, respectively,
satisfy that the starting Latin square $L_0$ is well-behaved
with respect to $A'$ and $P'$. 

If $J$ is a subset of cells of an array $S$,
$S'$ is the array obtained from $S$ by applying
$\rho$ to the rows of $S$ and $\theta$ to the columns of $S$,
then $\rho(\theta(J))$ denotes the set of cells
in $S'$ that $J$ are mapped to under $\rho$ and $\theta$.

Following \cite{AndrenOdd}, we shall for convenience 
in fact prove that there are permutations
$\sigma, \tau$, such that if $S$ is the
Latin square obtained from $L_0$ by applying $\sigma$ to the
rows and $\tau$ to the columns of $L_0$, then 
$L_0$, $S$, $A$ and $P$ satisfy the following:

\begin{itemize}

	\item[(a')] all cells in $S$ except for $3n+7$
	are in at least 
	$\lfloor n/2 \rfloor - \varepsilon n$ allowed strong intercalates;
	
	\item[(b')] for a collection
	$J_1, \dots, J_{3n}$ of $3n$ given $n$-sets of
	cells in 
	$L_0$, each $J_i$ satisfies
	that the corresponding $n$-set
	$\sigma(\tau(J_i))$
	of cells in $S$
	has at most $c(n)$ conflicts with $A$;

\item[(c')] for a collection $J_1,\dots, J_n$ of  $n$ given $n$-sets in $L_0$,
each $J_i$ satisfies that  	the corresponding $n$-set $\sigma(\tau(J_i))$ of
cells in $S$ contains at most $c(n)$ prescribed cells;

\item[(d')] for a collection $J_1,\dots, J_n$ of  $n$ given $n$-sets in $L_0$
and each symbol $s \in \{1,\dots, n\}$, each $J_i$ satisfies that
the corresponding $n$-set $\sigma(\tau(J_i))$ of cells in $S$ contains 
at most $c(n)$ cells such that $s$ is in the corresponding cell of $A$.

\end{itemize}

It is straightforward to deduce that if the above conditions
hold, then if we  denote by $P'$ and $A'$ the arrays
obtained from $P$ and $A$, respectively,
by applying the inverses of $\sigma$ and $\tau$
to the rows and columns, respectively, of $P$ and $A$,
then $L_0$ is well-behaved with respect to $P'$ and $A'$;
if (a') holds, then clearly (a) is true 
for $L_0, P'$ and $A'$ as well;
and if (b') is true, then by taking the $3n$ $n$-sets
in (b') to be the sets of the cells in a 
particular row or column, or containing a particular symbol,
we deduce that (b), (c) (d) hold for $L_0, A'$ and $P'$.
That (e) and (f) are true, are deduced similarly from the
fact that (c') and (d') hold.

Now, let $L_0$ be the starting Latin square defined above,
and let $\sigma$ and $\tau$ be two permutations chosen
independently and uniformly at random from all
$n!$ permutations of $\{1,\dots, n\}$.
Denote by $S$ a random Latin square
obtained from $L_0$ by applying $\sigma$ to the rows
of $L_0$ and $\tau$ to the columns of $L_0$.

\begin{lemma}
	\label{lem:manyintercalates}
		If 
$$\left( \frac{2 \beta}{\varepsilon- 2\beta} \right)^{\varepsilon- 2\beta}    
\left( \frac{1}{1 - 2\varepsilon + 4\beta} \right)^{1/2- \varepsilon + 2\beta}
< 1,$$ 
 and $\varepsilon > 2 \beta$,
then	the probability that $S$ fails condition (a')
		tends to $0$ as $n\to \infty$.
\end{lemma}
\begin{proof}
We estimate the number of pairs $(\sigma,\tau)$ such that there 
is at least one cell, except the $3n+7$ excluded, 
which does not belong to at least $\lfloor \frac{n}{2} \rfloor - \varepsilon n$ 
allowed strong intercalates. 

	There are at most $n^2$ cells 
	that can belong to too few allowed strong intercalates in $S$;
	choose such a cell $(r',c')_S$.
	Next, we fix $\tau$ by choosing one out of 
	$n!$ possible 
	permutations for $\tau$. Assume
	that $c'=\tau(c)$.
	
	With $\tau$ fixed, we now count in how many ways
	$\sigma$ can be chosen so that the cell $(r',c')_S$
	belongs to less than $\lfloor \frac{n}{2} \rfloor - \varepsilon n$
	allowed strong intercalates.
	
	There are $n$ choices for a row $r$ in $L_0$
	so that $\sigma(r) = r'$. This choice partitions the
	rows of $L_0$ into two sets: the set $Q$ of rows $r^*$
	for which 
	$\{(r,c)_{L_0}, (r, c^*)_{L_0}, (r^*,c)_{L_0}, (r^*, c^*)_{L_0} \}$
	is a strong intercalate in $L_0$ for some $c^* \neq c$, and its
	complement $\bar Q$. Note that $|Q|= \lfloor n/2 \rfloor$.

	Note further that choosing the row $r$
	in $L_0$ so that $\sigma(r) = r'$, 
	determines the value of \(s=L_0(r,c)\).
	When row $r$ and thus $S(r',c')$ is fixed,
	there are at most $\beta n$ columns
	$c_1$ such that $S(r',c') \in A(r',c_1)$.
	Furthermore, at most $\beta n$ columns $c_2$
	satisfy $S(r',c_2) \in A(r', c')$.
	Consequently, if there are less than $\lfloor n/2 \rfloor - \varepsilon n$
	allowed strong intercalates containing $(r',c')_S$ in $S$,
	then there have to be at least $\varepsilon n - 2\beta n$
	strong intercalates in $S$ containing $(r',c')_S$
	that are not allowed because swapping on them would cause
	a conflict in another row than $r'$. 
	(Note that $(\varepsilon- 2 \beta) > 0$
	by assumption.)
	The number of ways of choosing $\sigma$ so that in $S$
	at least $(\varepsilon- 2 \beta)n$ of the strong intercalates
	containing $(r', c')_S$ satisfy this condition can be estimated
	in the following way:
	Let $W$ be the set of rows in $S$ to which $\sigma$ maps $Q$.
	There are $\binom{n-1}{\lfloor n/2 \rfloor}$
	ways of choosing $W$. After choosing $W$ we can now choose how
	$\sigma$ acts on $\bar Q \setminus \{r\}$ in any of 
	the at most $(\lceil n/2 \rceil)!$
	possible ways. Next, we choose a subset $V \subseteq Q$
	of size $\lceil (\varepsilon- 2 \beta)n \rceil$.
	If we set $p(n) = \lceil (\varepsilon- 2 \beta)n \rceil$, 
	then this can be done in at most
	$\binom{\lfloor n/2 \rfloor}{p(n)}$
	ways.
	
	Now we define a bipartite graph $G_1$
	with parts $Q = \{r_1, \dots, r_{\lfloor\frac{n}{2}\rfloor}\}$ 
	and $W = \{r'_1, \dots, r'_{\lfloor\frac{n}{2}\rfloor}\}$.
	Include an edge between $r_i$ and $r'_j$ in $G_1$ 
	if and only if
	
	\begin{itemize}
		
		\item $r_i \notin V$, or
		
		\item $r_i \in V$ and $\sigma(r_i) = r'_j$ implies
		that the strong intercalate
		$$\{ (r',c')_S, (r', \tau(c_q))_S, (r'_j, c')_S, (r'_j, \tau(c_q))_S\}$$
		is not allowed in $S$ because swapping on it yields a
		conflict in row $r'_j$, where $c_q$ is the unique column such that
		$$\{ (r,c)_{L_0}, (r, c_q)_{L_0}, (r_i, c)_{L_0}, (r_i, c_q)_{L_0}\} $$
		is a strong intercalate in $L_0$.
	
	\end{itemize}

	A perfect matching in $G_1$ corresponds to  choosing $\sigma$
	so that at least $(\varepsilon  - 2\beta )n$ 
	strong intercalates
	in $S$ containing $(r',c')_S$ are not allowed because swapping
	on them yields conflicts on other rows than $r'$.
	
	The degree of  a vertex in $V$ is at most $2 \beta n$,
	because the symbols $L(r,c)$ and $L(r, c_q)$
	each occur at most $\beta n$ times in columns
	$\tau(c_q)$ and $\tau(c) = c'$ in $A$, respectively.
	The degree of a vertex in $Q \setminus V$ is 
	$\lfloor n/2 \rfloor$.
	Hence, by Corollary \ref{cor:bipbregman},
	there are at most
	$$ (\lfloor 2 \beta n \rfloor !)^{\frac
	{p(n) } { \lfloor 2 \beta n \rfloor} } (\lfloor n/2 \rfloor !)^{\frac
	{\lfloor n/2 \rfloor - p(n)}{\lfloor n/2 \rfloor}}$$
	perfect matchings in $G_1$.
	
	So the probability that $S$ fails condition (a') is at most
\begin{align*}
 & \frac{n^2 n! n \binom{n-1}{\lfloor n/2 \rfloor} 
\lceil n/2 \rceil!
	\binom{\lfloor n/2 \rfloor}{p(n)}
	(\lfloor 2 \beta n \rfloor !)^{\frac
	{p(n) } { 2 \beta n} } (\lfloor n/2 \rfloor !)^{\frac
	{\lfloor n/2 \rfloor - p(n)}{\lfloor n/2 \rfloor}}}
{(n!)^2} \\
 \leq & \frac{n^3  (\lfloor 2 \beta n \rfloor !)^{\frac
	{p(n) } { \lfloor 2 \beta n \rfloor} } 
	(\lfloor n/2 \rfloor !)^{\frac
	{\lfloor n/2 \rfloor - p(n)}{\lfloor n/2 \rfloor}} }
{p(n)! (\lfloor n/2 \rfloor - p(n))! }.
\end{align*}
By applying Stirling's formula, this expression tends to zero
as $n \to \infty$,
if
$$\left( \frac{2 \beta}{\varepsilon- 2\beta} \right)^{\varepsilon- 2\beta}    
\left( \frac{1}{1 - 2\varepsilon + 4\beta} \right)^{1/2- \varepsilon + 2\beta}
< 1,$$
which holds by assumption.
\end{proof}

\begin{lemma}
\label{lem:nset}

	Let $$J =\{(r_1,c_1)_{L_0}, \dots, (r_n,c_n)_{L_0}\}$$
	be a set of $n$ cells in $L_0$ and denote by
	$$J' = \{(r'_1,c'_1)_S, \dots, (r'_n, c'_n)_S\},$$
	where $\sigma(r_i) = r'_i$ and $\tau(c_i) = c'_i$,
	$i=1,\dots, n$. Then the following holds:
	
	\begin{itemize}
	
		\item[(i)] the probability that $J'$
			has at least $c(n)$ conflicts with $A$ is at most
			$$C n^a \left( \frac{\beta(n-c(n))}{c(n)} \right)^{c(n)}
				\left( \frac{n}{n-c(n)} \right)^{n},$$
			where $C$ and $a$ are some positive constants.
		\item[(ii)] the probability that $J'$ contains at least
			$c(n)$ prescribed cells is at most
			$$C n^a \left( \frac{\alpha(n-c(n))}{c(n)} \right)^{c(n)}
				\left( \frac{n}{n-c(n)} \right)^{n},$$

\item[(iii)] for a given symbol $s$,  the probability that $J'$ contains at least $c(n)$
cells such that the corresponding cell in $A$ contains $s$ is at most
$$C n^a \left( \frac{\beta(n-c(n))}{c(n)} \right)^{c(n)}
			\left( \frac{n}{n-c(n)} \right)^{n}.$$
\end{itemize}
	
\end{lemma}
	
\begin{proof}
We first prove (i).
We estimate the number of pairs $(\sigma, \tau)$
such that at least $c(n)$ cells from $J'$ are conflict cells
with $A$. There are $n!$ ways of choosing the permutation
$\sigma$. Fix such a permutation $\sigma$
 and suppose that $\sigma(r_i) = r'_i, i =1,\dots,n$.

Let $K$ be a subset of $J$ such that
$|K| = c(n)$ and all cells in $K$ are mapped
to conflict cells by $(\sigma, \tau)$. Such a
set $K$ can be chosen in $\binom{n}{c(n)}$ ways.
The number of ways of choosing $\tau$ so that
$(r'_i, c'_i)_S$ is a conflict cell whenever
$(r_i, c_i)_{L_0} \in K$ can be estimated by considering
a bipartite graph $G_2$ as follows: the parts
of $G_2$ are $J$ and $\{1,\dots, n\}$ and
there is an edge between
$(r_i, c_i)_{L_0} \in J$ and $j \in \{1,\dots, n\}$ if

\begin{itemize}

\item $(r_i, c_i)_{L_0} \notin K$, or

\item $(r_i, c_i)_{L_0} \in K$ and $L_0(r_i, c_i) 
\in A(r'_i, j)$.

\end{itemize}
Note that if $(r_i, c_i)_{L_0} \in K$
then the degree of $(r_i, c_i)_{L_0}$ in $G_2$
is at most $\beta n$, because the symbol
$L_0(r_i, c_i)$ occurs at most $\beta n$ times
in row $r'_i$ in $A$. If $(r_i, c_i)_{L_0} \notin K$,
then the degree of $(r_i, c_i)_{L_0}$ is $n$.

A perfect matching in $G_2$ corresponds to a choice of
$\tau$ so that all cells in $K$ are mapped to conflict
cells of $S$. By Corollary \ref{cor:bipbregman},
the number of perfect matchings in $G_2$
is at most
$$( \lfloor \beta n \rfloor! )^{\frac{c(n)}
	{\lfloor \beta n \rfloor}}
	( n!)^{\frac{n-c(n)}{n}}.$$
So the probability that $J'$
has at least $c(n)$ conflicts with $A$
is at most
	\begin{align*}
	& \frac{n!\binom{n}{c(n)}
	( \lfloor \beta n \rfloor! )^{\frac{c(n)}
	{\lfloor \beta n \rfloor}}
	( n!)^{\frac{n-c(n)}{n}}}{ (n!  )^2} \\
	 = & \frac{
	( \lfloor \beta n \rfloor ! )^\frac
	{c(n)}{\lfloor \beta n \rfloor} 
	(n! )^{\frac{n-c(n)}{n}}}{c(n)! (n-c(n) )!} \\
	 = & C n^a \left( \frac{\beta(n-c(n))}{c(n)} \right)^{c(n)}
			\left( \frac{n}{n-c(n)} \right)^{n},
	\end{align*}
	where $C$ and $a$ are some positive constants

The proof of (ii) is almost identical
to the proof of (i), the only difference
is that one uses the property that
each row in $P$ has at most $\alpha n$
non-empty cells, instead of the property
the each symbol occurs at most $\beta n$
in each row of $A$. The details are omitted.

The proof of (iii) is also almost identical to the
proof of (i) above except that one uses
the property that a fixed symbol $s$ occurs
at most $\beta n$ times in each row of $A$. 
Here as well, the details are omitted.
\end{proof}

\begin{lemma}
\label{lem:b'}
	If 
	$$\alpha < \frac{c(n)}{n-c(n)}
	\left(\frac{n-c(n)}{n}\right)^{\frac{n}{c(n)}}, \quad
	\beta < \frac{c(n)}{n-c(n)}
	\left(\frac{n-c(n)}{n}\right)^{\frac{n}{c(n)}},$$
	then the probability that $S$ fails condition 
	(b'), (c') or (d')
	tends to $0$ as $n \to \infty$.
\end{lemma}

\begin{proof}
	Let $J_1, \dots, J_{3n}$ be $3n$ given $n$ sets of cells
	in $L_0$. By part (i) of Lemma \ref{lem:nset},
	the probability that $J_i$ has at least
	$c(n)$ conflicts with $A$ is at most
	$$p_1 = C n^a \left( \frac{\beta(n-c(n))}{c(n)} \right)^{c(n)}
			\left( \frac{n}{n-c(n)} \right)^{n},$$
	where $C$ and $a$ are some positive constants.
	Since 
	$3np_1 \to 0$ as $n \to \infty$,
	it follows that the probability 
	that $S$ fails condition (b') tends to zero as 
	$n \to \infty$.
That the probability that $S$ fails condition (c')
or (d') tends to zero, can be proved similarly
using part (ii) and (iii) of Lemma \ref{lem:nset}.
\end{proof}

We conclude from the preceding lemmas
that there are permutations $(\sigma, \tau)$
such that if $S$ is obtained from $L_0$
by applying $\sigma $ to the rows of $L_0$,
and $\tau$ to the columns of $L_0$,
then $S$ satisfies (a'), (b'), (c') and (d').
Hence, if we denote by 
$A'$ and $P'$
the arrays obtained from $A$ and $P$, respectively,
by applying $\sigma^{-1}$ to the rows and 
$\tau^{-1}$ to the columns, then $L_0$ is
well-behaved with respect to $A'$ and $P'$.

\bigskip
\bigskip

\noindent
{\bf Step III:}
	By the preceding step, we may assume that the starting Latin square
	$L_0$ is well-behaved
	with respect to the array $A'$ and the
	PLS $P'$ defined above.
	We shall define a PLS $R$, such that a cell
	in $R$ is non-empty if and only if the corresponding cell of
	$L_0$ is a conflict cell with $A'$ and the corresponding
	cell of $P'$ is empty.

	Consider a bipartite graph $G_3$, where 
	the rows and columns of $L_0$ are the
	 vertices of the partite sets of $G_3$,
	and the conflict cells of $L_0$ defines the edge set
	of $G_3$, i.e. there is an edge between two vertices
	in $G_3$ if the corresponding cell of $L_0$
	is a conflict with $A'$.

	We want to find a proper $n$-coloring of $E(G_3)$ satisfying
	that if $R$ is the PLS corresponding 
	to this edge
	coloring of $G_3$ (by taking the partite sets of $G_3$ to be the
	rows and columns of $R$, and the colored edges of $G_3$ as the
	non-empty cells of $R$), then $R$ contains at most
	$c(n)$ entries in each row and column and each symbol in
	$R$ is used at most $f(n)$ times.
	This means that taking $P$ and $R$ together, they
	form a PLS
	where each row and column is used at most $\alpha n +c(n)$ times
	and each symbol is used at most
	$\alpha n + f(n)$ times.

	We may assume that there is no conflict cell in $L_0$
	such that the corresponding cell in $P'$ is non-empty,
	because then we just remove this cell from the set
	of conflict cells.
	We define a list assignment $\mathcal{L}$ for $G_3$ by
	for every symbol (color) $c \in \{1,\dots,n\}$ and every edge $e=ij$
	including $c$ in $\mathcal{L}(e)$ if and only if 
	$c \notin A'(i,j)$
	and $c$ does not appear in row $i$ or column $j$ in $P'$.
	Clearly, $$\mathcal{L}(e) \geq n-\beta n- 2\alpha n,$$
	for every edge $e$ of $G_3$.
	Our goal is to find an $\mathcal{L}$-coloring 
	$\phi$ of $E(G_3)$ such that
	each color appears on at most $f(n)$ edges. 
	Such a coloring of $G_3$ corresponds to a 
	PLS $R$ satisfying the
	conditions stipulated above.
	
	The maximum degree in $G_3$ is $c(n)$, because each
	row and column in $L_0$ contains at most $c(n)$ conflict cells
	(by condition (b) and (c) above).
	Since $$c(n) \leq n - \beta n - 2 \alpha n,$$
	there is an $\mathcal{L}$-coloring $\varphi$ of $G_3$
	by Theorem \ref{th:Galvin}.
	Suppose that there is some {\it dense} color $c_0$ in 
	$\varphi$,
	i.e. a color that is used
	more than $f(n)$ times in $\varphi$. We will define
	an $\mathcal{L}$-coloring $\varphi'$ so that for some edge
	$e$ with $\varphi(e) = c_0$, $e$ is colored with some
	non-dense color in $\varphi'$. By iterating this process,
	we obtain the required coloring $\phi$.
	
	So suppose that $\varphi(e) = c_0$. The number of dense
	colors in $\varphi$ is at most $n c(n) /f(n)$.
	Moreover, there are at most $2 c(n)$
	distinct colors that are used on edges which are adjacent to $e$.
	Hence, we can define $\varphi'$ from $\varphi$
	by selecting a new color for $e$ 
	so that the resulting coloring is proper
	if
	$$n - \beta n - 2 \alpha n - 2 c(n)- \frac{n c(n)} {f(n)} \geq 1,$$
	which holds by assumption.
	We conclude that the required coloring $\phi$ exists
	and thus also the required PLS $R$.
	\bigskip
	
	Let $\hat P$ be the PLS obtained by putting $P'$ and $R$ together.

	The PLS $\hat P$ satisfies the following
	\begin{itemize}
			
			\item[(a'')] $\hat P$ contains at most  $\alpha n + c(n)$ entries
			in each row or column;
			
			\item[(b'')] each symbol is used at most
			$\alpha n + f(n)$ times in $\hat P$.
	
	\end{itemize}

	Furthermore, since $L_0$ is well-behaved with respect to $A'$ and $P'$,
	it satisfies the following conditions with respect to $A'$ and $\hat P$:
	
	\begin{itemize}

	\item[(c'')] each cell in $L_0$ (except for $3n+7$)
	belongs to at least $\lfloor n/2 \rfloor - \varepsilon n$ 
	allowed strong intercalates;
	
	\item[(d'')] each row and column of $L_0$ contains at most
	$\alpha n + c(n)$ prescribed cells;
	
	\item[(e'')] for each symbol $s$, there are at most
	$2 c(n)$ prescribed cells in $L_0$ with entry $s$;
	
	\item[(f'')] for each pair of symbols $s_1, s_2$,

	there are at most $c(n)$ cells in $L_0$ with entry $s_1$ such
	that $s_2$ appears in the corresponding cell in $A'$.
	
	\end{itemize}

\bigskip
\bigskip

\noindent
{\bf Step IV:}
Let $\hat P$ be the PLS obtained in the previous step,
and $A'$, $P'$ and $L_0$  as above. In this section, all prescribed
cells of a Latin square is taken with respect to $\hat P$.

Let $L$ be a Latin square obtained from the starting Latin square
$L_0$ by performing a sequence of trades. We say that a cell $(i,j)_L$ in $L$ is 
{\em $L$-disturbed}
if $(i,j)_L$ appears in a trade which is used for obtaining $L$ from $L_0$,
or if $(i,j)_{L_0}$ is one of the original at most $3n+7$ cells in
$L_0$ that do not belong to at least $\lfloor n/2 \rfloor - \varepsilon n$
allowed strong intercalates in $L_0$.

Let $L$ be a Latin square obtained from $L_0$
by a performing a sequence of trades. For a constant $d>0$, we say
that
a row or column $r$ or symbol $s$ is {\em $d$-overloaded}
if more than $dn$ entries in row or column $r$
or with symbol $s$ has been involved in the trades
that has transformed $L_0$ into $L$.

In this step we describe a modified variant
of the machinery developed in \cite{Bartlett}
for completing sparse partial Latin squares.
The main difference is that we have to make sure
that no trades will cause any ``new'' conflict cells with $A'$.
In particular, the intercalates that we will swap on 
will be allowed with respect to $A'$. 
Another difference is that all symbols used in the trade
created by Lemma \ref{lem:22} below (our version of Lemma 2.2
in \cite{Bartlett}) are not $d$-overloaded.
Apart from these differences,
the proofs in this section are almost identical to the ones
in \cite{Bartlett}, so in general, proofs are sketched,
rather than given in full detail. Also, we omit many 
verifications which can be done exactly as in \cite{Bartlett}
(or \cite{AndrenOdd} in some cases).

We will define a sequence of Latin squares $L_0, \dots, L_q$,
where $L_i$ is obtained from $L_{i-1}$, $i=1,\dots, q-1$,
by performing some trade $T_i$. The trade $T_i$ will contain
(at least) one prescribed cell $(r,c)_{L_{i-1}}$ such that
$L_{i-1}(r,c) \neq \hat P(r,c)$,
$L_i(r,c) = \hat P(r,c)$, and, furthermore, all conflict cells of $L_{i}$
will be prescribed cells $(r', c')$ 
such that $\hat P(r',c') \neq L_i(r',c')$, i.e.
the trade $T$ does not create any ``new'' conflict cells.

In the following we shall refer to the ``lower half'' and ``upper half''
of an array $L$; by these expressions we mean the subarray of $L$
consisting of the first $\lfloor n/2 \rfloor$
rows of $L$ and the subarray consisting of the
last $\lceil n/2 \rceil$ rows of $L$, respectively.
We also assume that if $n$ is odd, then the row and column of
$L_0$ where no cells are in at least
$\lfloor n/2 \rfloor$ strong intercalates are the last
row and column of $L_0$, respectively.
	
	The following lemma is essentially a strengthened variant
	of Lemma 2.2. in \cite{Bartlett}.

	\begin{lemma}
	\label{lem:22}
		Let $L_0$, $\hat P$ and $A'$ be as above.
		Suppose that $L$ is an $n\times n$
		Latin square obtained from $L_0$ by performing
		some sequence of trades on $L_0$, and that
		at most $k n^2$ cells in $L$ are $L$-disturbed,
		for some constant $k> 0$. 
		
		Let $\{t_1,\dots,t_a\}$ be a set of $a$ symbols
		from $L$.

		If
		$$\left\lfloor \frac{n}{2} \right\rfloor - 
		2 \varepsilon n -  6 dn  - 5\frac{k}{d} n - 4 
		\alpha n - 8 c(n) -
		3a - 3 \beta n > 6$$
		then for any row $r_1$ of $L$ and all but at most
		\begin{itemize}
		
			\item $2 \frac{k}{d} n + \alpha n + c(n) + a$ choices of $c_1$, and
		
			\item $ a +1 + 4c(n) + 2\beta n  + 4 \frac{k}{d} n + 2\alpha n + 2dn$ 
			choices of $c_2$,
		
		\end{itemize}
	there is a set of cells $T$, 

	such that if we denote by $L'$ the Latin square obtained from $L$
	by performing a trade on $T$, then $L'$
	satisfies the following:
	\begin{itemize}

		\item the trade $T$ uses only
		symbols that are not $d$-overloaded;
		
		\item no prescribed cells of $L$ are in $T$;
		
		\item $L$ and $L'$ differs on at most $16$ cells
		(i.e. $T$ uses at most $16$ cells);
		
		\item no cell with entry $\{t_1,\dots,t_a\}$
		in $L$ is in $T$;
		
		\item $L'(r_1,c_1) = L(r_1,c_2)$ and 
		$L'(r_1,c_2) = L(r_1, c_1)$;
		
		\item if there is a conflict of $L'$ with $A'$,
		then the corresponding cell of $L$ is also
		a conflict with $A'$.

	\end{itemize}
	\end{lemma}
	
	\begin{proof}
		Consider a given row $r_1$. We choose
		a column $c_1$ in $L$, such that:
		
		\begin{itemize}
		
		\item Column $c_1$ is not
		$d$-overloaded, and the symbol $s_1 = L(r_1, c_1)$
		is not overloaded. This eliminates $2\frac{k}{d} n$ choices. 
		
		\item The cell $(r_1,c_1)_L$ is not a prescribed
		cell. This eliminates at most $\alpha n + c(n)$
		choices.
		
		\item The symbol $s_1$ is not one of
		$\{t_1,\dots,t_a\}$. This eliminates at most $a$
		choices.
		
		\end{itemize}
		Summing up, we have at least 
		$$n - 2 \frac{k}{d} n - \alpha n - c(n)- a$$
	choices for $c_1$; by assumption this expression
	is greater than zero, so we fix such a column $c_1$.

	Next, we choose a column $c_2$ in $L$ so that
	the following properties hold:

	\begin{itemize}
		
		\item $c_2 \neq c_1$ and 
		$s_2 = L(r_1, c_2) \notin A'(r_1,c_1)$ and
		$s_1 \notin A'(r_1, c_2)$.
		This excludes at most $1 +  2 \beta n$ choices for $c_2$.
		
		\item Column $c_2$ is not
		$d$-overloaded, and the symbol $s_2 = L(r_1, c_2)$
		is not $d$-overloaded. This eliminates $2\frac{k}{d} n$ choices. 
		
		\item The cell $(r_1,c_2)_L$ is not a prescribed
		cell. This eliminates at most $\alpha n + c(n)$
		choices.
		
		\item The cell $(r_3, c_1)_L$
		in column $c_1$ in $L$ containing
		$s_2$ is not $L$-disturbed, 
		and the cell $(r_4, c_2)_L$ 
		in column $c_2$ in $L$ containing
		$s_1$ is not $L$-disturbed.
		Since 
		neither the column $c_1$ nor the symbol $s_1$
		is $d$-overloaded, this
		excludes at most $2dn$ choices. We also require
		that the cells 
		$(r_3, c_1)_L$  and $(r_4, c_2)_L$ 
		are not prescribed, which
		excludes an additional $ 3c(n) + \alpha n$ choices.
		
		\item The rows $r_3, r_4$ are not $d$-overloaded.
		This eliminates at most $2 \frac{k}{d} n$ choices.
		
		\item $s_2 \notin \{t_1,\dots, t_a\}$.
		This excludes at most $a$ choices.

	\end{itemize}
		Summing up, we have at least
		$$n  - 4 c(n) - 2 \beta n  - 4 \frac{k}{d} n - 2 \alpha n - 2dn	- a -1$$
	choices for $c_2$; by our assumptions this expression is greater
	than zero, and so we fix such a column $c_2$ in $L$.
	
	\bigskip

	\noindent
	{\bf Case 1.}
	{\em Both of the rows $r_3$ and $r_4$  lie
	either in the upper half or in the lower half of
	the Latin square $L$ (and thus in $L_0$):}
	
	We may assume that $r_3 \neq r_4$, since otherwise
	we may swap on the intercalate consisting of all
	hitherto considered cells, and are done.
	Assuming $r_3 \neq r_4$, we now proceed as follows:
	
	For the trade in Case 1, we shall construct two disjoint allowed
	strong intercalates
	
	$$C_1 = \{(r_3,c_1)_L, (r_2, c_1)_L, (r_2, c_4)_L, (r_3, c_4)_L\}$$
	and
	$$C_2 = \{(r_4,c_2)_L, (r_2, c_2)_L, (r_2, c_3)_L, (r_4, c_3)_L\},$$
	containing the  cells
	$(r_3, c_1)_L$  and $(r_4, c_2)_L$, respectively.
	Since these two cells are not $L$-disturbed,
	they agree with $L_0$, and
	the corresponding cells in $L_0$ are both in at least
	$\lfloor n/2 \rfloor - \epsilon n$ 
	allowed strong intercalates in $L_0$, and
	since they lie in ``the same half'' of $L_0$, there are at least
	$\lfloor n/2 \rfloor - 2 \varepsilon n$ such pairs of 
	allowed strong intercalates in $L_0$
	containing a common row $r_2$.
	We further require that:
	
	\begin{itemize}
	
		\item None of the cells
		$(r_2, c_1)_L, (r_2, c_2)_L, (r_2, c_3)_L, (r_2, c_4)_L, 
		(r_3, c_4)_L$,
		or $(r_4, c_3)_L$ are $L$-disturbed. Because
		none of the rows $r_3, r_4$, the columns $c_1, c_2$ or
		the symbols $s_1, s_2$ are overloaded, this excludes
		at most $6 d n$ choices. Note that this condition
		ensures that all cells of $C_1$ and $C_2$ have the same
		entry in $L$ as the corresponding cells of $L_0$.
		
		\item None of the cells above are prescribed. This excludes
		at most $4(\alpha n + c(n))$ + $4 c(n)$ choices.
		
		\item Neither $s_3=L(r_2, c_1)$ or $s_4 = L(r_2, c_2)$ is 
		in $\{t_1, \dots, t_a\}$.
		This eliminates at most $2 a$ choices.
		
		\item The symbols $s_3$ and $s_4$ are not $d$-overloaded.
		This excludes at most $2 \frac{k}{d} n$ choices.
		
		\item $s_1 \notin A'(r_2, c_1)$ and $s_2 \notin A'(r_2, c_2)$.
		This eliminates at most $2 \beta n$ choices.
	\end{itemize}
	
	Summing up we have at least
	$$\left\lfloor \frac{n}{2} \right\rfloor 
	- 2 \varepsilon n - 6 d n - 4\alpha n - 8 c(n)  - 2 a -  
	2 \frac{k}{d} n-  2 \beta n$$
	choices for the required intercalates $C_1$ and $C_2$. 
	Since this expression is greater than zero,
	we choose two such disjoint intercalates, $C_1$ and $C_2$.
	
	By swapping on $C_1$ and $C_2$ we obtain a Latin square
	$L^{(1)}$. Note that the set
	$$\{(r_1,c_1)_{L^{(1)}}, (r_1,c_2)_{L^{(1)}}, 
	(r_2,c_1)_{L^{(1)}}, (r_2,c_2)_{L^{(1)}} \}$$ is an allowed intercalate
	in $L^{(1)}$ and by swapping on this intercalate we obtain the
	required Latin square $L'$. This completes the proof of the lemma in Case 1.

	\bigskip
	
	\noindent
	{\bf Case 2.}
	{\em One of rows $r_3$ and $r_4$ occur
	in the upper half and the other one in the lower half of 
	the Latin square $L$:}

	Suppose without loss of generality that $r_3$
	lies in the lower half of $L$ and that $r_4$ lie
	in the upper half of $L$.
	We will construct several intercalates for the trade
	in Case 2.
	To begin with we construct an allowed strong intercalate
	
	$$C_3 = \{(r_4, c_2)_L, (r_2, c_2)_L, (r_2, c_3)_L, (r_4, c_3)_L \},$$
	containing the cell $(r_4, c_2)_L$ 
	such that the following holds:
	
	\begin{itemize}
	
	\item None of the cells $(r_2, c_1)_L, (r_2, c_2)_L, (r_2, c_3)_L,
	(r_4, c_3)_L$ are $L$-disturbed. Because
	neither row $r_4$, nor columns $c_1, c_2$, nor symbols $s_1$,
	are $d$-overloaded. This eliminates at most $4 dn$ choices.
	
	\item If $(r_2, c_4)_L$ is the cell in row $r_2$ containing $s_2$,
	then $(r_2, c_4)_L$ and $(r_3, c_4)_L$ 
	are not $L$-disturbed. This excludes at most $2 dn$ choices.
	
	\item The symbols $s_3 = L(r_2, c_1)$, $s_4 = L(r_2, c_2)$
	and $s_5 = L(r_3, c_4)$ are not $d$-overloaded, as are not
	row $r_2$ or column $c_4$, and these new cells are disjoint
	from the ones previously included in our trade. This eliminates
	at most $5\frac{k}{d} n +2$ choices.
	
	\item None of the cells above are prescribed. This eliminates
	at most $4( \alpha n + c(n)) + 4 c(n)$ choices.
	
	\item None of the symbols $s_3, s_4, s_5$ is in $\{t_1, \dots t_a\}$.
	This eliminates at most $3a$ choices.
	
	\item $s_1 \notin A'(r_2, c_1)$, 
	$s_2 \notin A'(r_2, c_2) \cup A'(r_3, c_4)$.
	This eliminates at most $3 \beta n$ choices.
	
	\end{itemize}
	
	Since there are at least $\lfloor n/2 \rfloor - \varepsilon n$ 
	strong intercalates in $L_0$ containing $(r_4, c_2)_{L_0}$, 
	we have at least
	$$\left\lfloor \frac{n}{2} \right\rfloor - 
	\varepsilon n -  6 dn  -5\frac{k}{d} n -2 - 4 \alpha n - 8 c(n) -
	3a - 3 \beta n$$
	choices for the required intercalate $C_3$.
	By assumption this expression is greater than zero,
	and we choose such an intercalate $C_3$.
	
	Now, note that since $r_4$ lies in the upper half of
	$L$, $r_2$ lies in the lower half of $L$.
	Since $r_3$ also lies in the lower half of $L$, and
	none of the cells $(r_3,c_4)_L$, $(r_2, c_4)_L$, $(r_3, c_1)_L$
	and $(r_2, c_1)_L$
	are $L$-disturbed, and $L(r_3, c_1) = L(r_2, c_4) = s_2$,
	it follows that
	in $L_0$ there are at least
	$\lfloor n/2 \rfloor - 2 \varepsilon n $ pair of
	allowed disjoint strong intercalates
	
	$$C^{L_0}_4 = 
	\{(r_2, c_1)_{L_0}, (r_6, c_1)_{L_0}, (r_6, c_6)_{L_0}, (r_2, c_6)_{L_0} \}$$
	and
	$$C^{L_0}_5 = 
	\{(r_3, c_4)_{L_0} (r_5, c_4)_{L_0}, (r_5, c_5)_{L_0}, (r_3, c_5)_{L_0}\}$$
	containing
	$(r_2, c_1)_{L_0}$
	and $(r_3, c_4)_{L_0}$, respectively, and such that
	$L_0(r_6, c_1) = L_0(r_5, c_3)$.
	
	We choose such a pair 
	$$C_4 = \{(r_2, c_1)_{L}, (r_6, c_1)_{L}, (r_6, c_6)_{L}, (r_2, c_6)_{L} \}$$
	and
	$$C_5 = \{(r_3, c_4)_{L} (r_5, c_4)_{L}, (r_5, c_5)_{L}, (r_3, c_5)_{L}\}$$
	of intercalates in
	$L$ such that 
	the following holds:
	
	\begin{itemize}
	
	\item None of the cells in these intercalates 
	are $L$-disturbed. Because the columns $c_1, c_4$, 
	rows $r_2, r_3$ and
	symbols $s_3, s_5$ are not $d$-overloaded. This eliminates
	at most $6 dn$ choices.
	
	\item None of the cells in these intercalates are prescribed.
	This eliminates at most $4( \alpha n + c(n)) + 4 c(n)$ choices.
	
	\item The symbol $s_6 = L(r_6, c_1) \notin \{t_1, \dots, t_a\}$, 
	and it is not overloaded.
	This eliminates $a+ \frac{k}{d} n$ choices.
	
	\item $s_6 \notin A'(r_3, c_1) \cup A'(r_2, c_4)$
	and $s_6 \notin \{s_1, s_2, s_3, s_4\}$.
	This eliminates at most $2 \beta n + 6$ choices.

	\end{itemize}
	
	Thus
	we have at least
	$$\left\lfloor \frac{n}{2} \right\rfloor - 2 \varepsilon n - 6 dn - 
	4 \alpha n - 8 c(n)  -a - \frac{k}{d} n - 6
	- 2 \beta n$$
	choices for the required intercalates $C_4$ and $C_5$ in $L$,
	and by assumption this expression is greater than zero.
	
	By swapping on the disjoint intercalates $C_3, C_4$ and $C_5$ 
	we obtain a Latin square $L^{(1)}$. Note that the
	set
	$$\{(r_2,c_1)_{L^{(1)}}, (r_2,c_4)_{L^{(1)}}, 
	(r_3,c_1)_{L^{(1)}}, (r_3,c_4)_{L^{(1)}} \}$$ is an intercalate
	in $L^{(1)}$ and by swapping on this intercalate we  obtain a
	Latin square $L^{(2)}$,
	in which the set 
	$$\{(r_1,c_1)_{L^{(2)}}, (r_1,c_2)_{L^{(2)}}, 
	(r_2,c_1)_{L^{(2)}}, (r_2,c_2)_{L^{(2)}} \}$$
	is an intercalate; by
	swapping on this intercalate we finally obtain the
	required Latin square $L'$. Moreover, it can be verified
	that $L'$ contains no conflicts with $A'$ that were not present in $L$.
	This completes the proof in Case 2.
\end{proof}

	Of course the analogous statement for columns is
	true as well:
	
	\begin{lemma}
	\label{lem:22column}
		Let $L_0$, $\hat P$ and $A'$ be as above.
		Suppose that $L$ is an $n\times n$
		Latin square obtained from $L_0$ by performing
		some sequence of trades on $L_0$, and that
		at most $k n^2$ cells of $L$ are $L$-disturbed,
		for some $k> 0$. 

		Let $\{t_1,\dots,t_a\}$ be a set of $a$ symbols
		from $L$.

		If
		$$\left\lfloor \frac{n}{2} \right\rfloor - 2 
		\varepsilon n -  6 dn  - 5\frac{k}{d}n - 4 \alpha n - 8 c(n) -
		3a - 3 \beta n > 6$$
		then for any column $c_1$ of $L$ and all but at most
		\begin{itemize}
		
			\item $2 \frac{k}{d}n + \alpha n + c(n) + a$ choices of $r_1$, and
		
			\item $ a +1 + 4c(n) + 2 \beta n  + 4\frac{k}{d}n + 2\alpha n + 2dn$ 
			choices of $r_2$,
		
		\end{itemize}
	there is a set of cells $T$ 
	such that if denote by $L'$
	the Latin square obtained from $L$ by performing a trade on $T$, then
	$L'$ satisfies the following:
	\begin{itemize}
		
		\item the trade $T$ uses only
		symbols that are not $d$-overloaded;
		
		\item no prescribed cells of $L$ are in $T$;
		
		\item $L$ and $L'$ differs on at most $16$ cells
		(i.e. $T$ uses at most $16$ cells);
		
		\item no cell with entry $\{t_1,\dots,t_a\}$
		in $L$ is in $T$;
		
		\item $L'(r_1, c_1) = L(r_2,c_1)$ and 
		$L'(r_2,c_1) = L(r_1, c_1)$;
		
		\item if there is a conflict of $L'$ with $A'$,
		then the corresponding cell of $L$ is also
		a conflict with $A'$.
		
	\end{itemize}
	\end{lemma}
		
		\bigskip
		
	The two above lemmas
	are used for exchanging the content of two cells in a Latin square; in the case
	of Lemma \ref{lem:22}, the cells
	are in positions $(r_1,c_1)$ and $(r_1, c_2)$, respectively.
	When using this lemma below, we shall refer to the cell in position
	$(r_1,c_1)$ as the ``first cell'' and the cell in position $(r_1,c_2)$
	as the ``second cell'', and similarly for Lemma \ref{lem:22column}.

	The two above lemmas can be used for proving the following,
	which essentially is a variant of Lemma 2.3 in \cite{Bartlett}.
	
	\begin{lemma}
	\label{lem:23}
		Let $L_0$, $\hat P$ and $A'$ be as above, 
		and $L$ be a Latin square
	obtained from $L_0$ by performing some sequence of 
	trades on $L_0$. Assume that
	at most $k n^2$ cells of $L$ are $L$-disturbed,
	where $k > 0$. 
	Suppose that
	$L$ has some prescribed cells where $L$ and $\hat P$ do not agree.
	In particular, for each symbol $s_i$,
	assume that at most $2 c(n) + 2 d(n)$ cells with 
	symbol $s_i$ are prescribed in $L$, and assume further that at most
	$4 \left(c(n) +  d(n) +  \alpha n + f(n) \right)$ cells in $L$ with
	symbol $s_i$ are $L$-disturbed.
	Let $(r_1, c_1)_L$ be a cell of $L$ such that
	$$L(r_1, c_1) = s_1 \text{ and } \hat P (r_1, c_1) = s_2,
	\quad s_1 \neq s_2.$$
	If 
	
	\begin{equation}
	\label{eq:majorcondition}
			n - 2\left(4 \frac{k+64/n^2}{d}n + 3 + 6 c(n) + 2 \beta n  +
			4 \frac{k}{d}n + 2\alpha n +   2 f(n) + 4 dn    \right) > 1,
	\end{equation}		
	then
	there is a set of cells $T$ in $L$, such that if
	we denote by $L'$ the Latin square obtained from $L$ by performing 
	a trade on $T$,
	then the following holds:
	
	\begin{itemize}
	
		\item $L'(r_1, c_1) = s_2$;
		
		\item $L'$ and $L$ disagree on at most $69$ cells;
		
		\item besides $(r_1, c_1)_L$,
		$L$ and $L'$ disagree on at most $2$ prescribed cells;
		
		\item if $L$ and $L'$ disagree on a prescribed cell $(r,c)_L$
		(where $r \neq r_1$ or $c_1 \neq c$),
		then $L'(r,c)$ is
		not $d$-overloaded and $L(r,c) \neq \hat P(r,c)$;
		
		\item the trade $T$ contains exactly two cells 
		with entry $s_1$
		in $L$, and at most four cells with entry $s_2$;
		
		\item except $s_1$ and $s_2$ the trade $T$ contains only
		cells with symbols that are not $d$-overloaded;
		
		\item if there is a conflict of $L'$ with $A'$,
		then the corresponding cell of $L$ is also
		a conflict with $A'$.

	\end{itemize}

	\end{lemma}
	
	\begin{proof}
			We shall  construct a trade
			from which we obtain $L'$ from $L$, where $L'$ and 
			$\hat P$ agree
			on the cell in position $(r_1, c_1)$. We will accomplish this
			by four succesive applications of Lemmas \ref{lem:22} and
			\ref{lem:22column}, similarly as how Lemma 2.2 in \cite{Bartlett} 
			is applied in that paper.
			In our application of Lemmas \ref{lem:22} and
			\ref{lem:22column}
			we will avoid
			the symbols $\{s_1, s_2\}$; so $a=2$ in the application
			of these lemmas.
			
			Let $(r_1, c_3)_L$ and $(r_3, c_1)_L$ be the cells
			in row $r_1$ and column $c_1$, respectively, that contains $s_2$.
			We want to choose a cell $(r_4,c_4)_L$
			such that $L(r_4, c_4)= s_1$, and if $r_2$ and $c_2$
			are the row and column, respectively, satisfying that
			$L(r_4, c_2) = s_2$
			and $L(r_2, c_4) = s_2$, then the following holds:

			\begin{itemize}
			
				\item The cells $(r_4, c_4)_L, (r_4, c_2)_L, (r_2, c_4)_L$
				are not prescribed cells. This eliminates at most
				$4 c(n) + 4 dn$ choices.
				
				\item The cell $(r_4, c_4)_L$ is not $L$-disturbed
				and $s_2 \notin A'(r_4, c_4)$. This eliminates
				at most $4 \left( c(n) +  d(n) + 
				 \alpha n + f(n) \right) + c(n)$ choices.
				
				\item $s_2 \notin A'(r_3,c_2) \cup A'(r_2,c_3)$
				and $s_1 \notin A'(r_4,c_1) \cup A'(r_1,c_4)$.
				This excludes at most $4 \beta n$ choices.

				\item The cells $(r_4, c_1)_L, (r_2, c_3)_L, (r_3, c_2)_L,
				(r_1, c_4)_L$ are all valid choices for the first cell
				to be changed in an application of Lemma \ref{lem:22} or
				\ref{lem:22column}.
				Since these Lemmas are applied four consecutive times
				this excludes at most 
				$$4\left(2 \frac{k+64/n^2}{d}n +\alpha n + c(n) + 2\right)$$
				choices. In particular, this implies that none of 
				these cells
				are prescribed or contains a $d$-overloaded symbol.

			\end{itemize}
			
			Thus we have at least
			$$n - 12 c(n) - 8 d(n)-   4 \alpha n - 4 f(n) - 4 \beta n
			- 4\left(2 \frac{k+64/n^2}{d}n   + 2\right)$$
			choices for such a cell $(r_4, c_4)_L$ containing symbol $s_1$. 
			We note
			that this expression is greater than zero by assumption,
			so we can indeed make the choice.

		Next, we want to choose a symbol $s_3$ in row $r_1$ and column $c_3$,
		such that the following holds:
		
		\begin{itemize}
			
			\item The cells with symbol $s_3$ in row $r_1$
			and column $c_3$ are both valid choices for the 
			second cell to be exchanged
			in an application of Lemma  \ref{lem:22} or \ref{lem:22column};
			this eliminates
			at most 
			$$2\left(4 \frac{k+64/n^2}{d}n + 3 + 4c(n) + 2\beta n  + 
			4\frac{k}{d}n + 2\alpha n + 2dn    \right)$$ choices.
			
			\item $s_3 \notin A'(r_1, c_3) \cup A'(r_2, c_4)$.
			This eliminates at most $2 \beta n$ choices.
			
			\end{itemize}
			
			Thus we have at least
			$$n- 2\left(4 \frac{k+64/n^2}{d}n + 3 + 4c(n) + 2 \beta n  + 
			4\frac{k}{d}n + 2\alpha n + 2dn    \right) - 2 \beta n$$
			choices for the symbol $s_3$.
			By assumption, this expression is greater
			than zero, so we can indeed choose such a symbol $s_3$.

			Similarly, we want to choose a symbol $s_4$ in row $r_3$
			and column $c_1$ such that the following holds:
		
			\begin{itemize}
			
			\item The cells with symbol $s_4$ in row $r_3$
			and column $c_1$ are both valid choices for the 
			second cell to be exchanged
			in an application of Lemma  \ref{lem:22} or \ref{lem:22column}; 
			this eliminates
			at most 
			$$2\left(4 \frac{k+64/n^2}{d}n + 3 + 4c(n) + 2\beta n  + 
			4\frac{k}{d}n + 2\alpha n + 2dn    \right)$$ choices.
			
			\item $s_4 \notin A'(r_4, c_2) \cup A'(r_3, c_1)$.
			This eliminates at most $2 \beta n$ choices.
			
			\end{itemize}
			Clearly, we have precisely the same number of
			choices for the symbol $s_4$ as for $s_3$.

			Now, by applying Lemmas \ref{lem:22} and \ref{lem:22column}
			to the cells
			$(r_1, c_4)_L$, and $(r_2,c_3)_L$, and the cells in column $c_3$
			and row $r_1$ containing symbol $s_3$, we may exchange the
			content of cells $(r_1, c_4)_L$, and $(r_2,c_3)_L$;
			and similarly for the cells
			$(r_4, c_1)_L$, $(r_3, c_2)_L$, and symbol $s_4$.
			
			Hence, by four succesive applications of 
			Lemmas \ref{lem:22} and
			\ref{lem:22column}
			we obtain a Latin square $L^{(1)}$, 
			such that the sets
			$$\{(r_3,c_1)_{L^{(1)}}, (r_4,c_1)_{L^{(1)}}, 
			(r_3,c_2)_{L^{(1)}}, (r_4,c_2)_{L^{(1)}} \}$$ and
			$$\{(r_1,c_3)_{L^{(1)}}, (r_1,c_4)_{L^{(1)}}, 
			(r_2,c_3)_{L^{(1)}}, (r_2,c_4)_{L^{(1)}} \}$$
			are disjoint intercalates. By swapping on these
			intercalates we obtain a Latin square $L^{(2)}$,
			where the set
			$$\{(r_1,c_1)_{L^{(2)}}, (r_1,c_4)_{L^{(2)}}, 
			(r_4,c_1)_{L^{(2)}}, (r_4,c_4)_{L^{(2)}} \}$$ 
			is an intercalate. By swapping on this
			intercalate we obtain the required Latin square $L'$.
	\end{proof}

	We will take care of all the prescribed cells of $L_0$
	by successively applying Lemma \ref{lem:23};
	using this lemma one can construct the Latin squares $L_0, L_1,\dots, L_{q}$,
	where $L_i$ is constructed from $L_{i-1}$ by an application of Lemma 
	\ref{lem:23}, and $L_{q}$ is an completion of $\hat P$, where 
	$q \leq n (\alpha n + c(n))$. Thus, in $L_i$  one more prescribed cell
	has the same entry as the corresponding cell in $\hat P$, 
	compared to $L_{i-1}$.

	Except for the cell $(r_1, c_1)_L$
	in Lemma \ref{lem:23},
	an application of Lemma \ref{lem:23} 
	will possibly change the content of two other 
	prescribed cells. However, it follows 
	that if this is the case, then in $L'$ each such prescribed
	cell contains a symbol
	that is not $d$-overloaded.
	Moreover, for each symbol $s$,
	$L_0$ has at most $2 c(n)$ prescribed cells containing $s$.
	Thus for each $i=1,\dots, q$,
	any symbol $s$ in $L_i$ occurs in at most $2 c(n) + 2 dn$ prescribed
	cells.
	Furthermore, each application of Lemma \ref{lem:23}
	to a prescribed cell $(r_1, c_1)_L$ with $L(r_1, c_1) = s$
	constructs a trade $T$ with exactly two cells containing
	symbol $s$. Hence, a symbol $s$ is used at most $2 (2 c(n) + 2 dn )$ times 
	in a trade where a prescribed cell has entry $s$.
	
	Note further that at most $\alpha n + f(n)$ cells $(r', c')_{\hat P}$ 
	in $\hat P$
	has entry $s$,
	and a trade $T$ constructed by an application of Lemma \ref{lem:23}
	for obtaining a Latin square $L'$ such that 
	$L'(r', c') = s$
	uses $4$ cells with entry $s$.
	
	Except for the cells mentioned in the preceding two paragraphs, any 
	other cells involved in a trade 
	created by an application of Lemma \ref{lem:23} contain
	symbols that are not $d$-overloaded.
	Hence, at most

	$$4 \left( c(n) + dn +  \alpha n + f(n)\right)$$
	distinct cells with a given symbol $s$ is 
	used in trades for constructing
	$L_q$ from $L_0$.
	
	Thus as long as 
	\eqref{eq:majorcondition}, $kn^2 \geq 69 n(\alpha n + c(n))$,
	and all the other conditions in the proof of Theorem \ref{th:main2} 
	hold, it follows
	that we can apply the last lemma iteratively for constructing the
	sequence $L_0, \dots, L_q$ of Latin squares, where $L_q$
	is a completion of $\hat P$ that avoids $A'$.
	This completes the proof of Theorem \ref{th:main2}.
	\end{proof}

\section{Random partial Latin squares and arrays}

In this section we prove Corollary \ref{cor:random}.
So let $P$ be a random PLS from the probability space $\mathcal{P}(n,p)$ defined above; and let $A$ be a random array
where each cell $(i,j)_A$ of $A$ a set $A(i,j)$ of size $m = m(n)$ by choosing each set uniformly at random
from all $m$-subsets of $[n]$. 
Assume further that no entry of $A$ occurs in the corresponding cell of $P$.
We need to prove that there are constants $\rho_1$ and $\rho_2$ such that
if $p< \rho_1$ and $m \leq \rho_2 n$, and where we for any cell of $A$ containing an entry that occurs in the corresponding 
cell of $P$, remove that entry from $A$, then with probability tending to $1$, there is a completion of $P$ that avoids $A$. 
We will use simple first moment calculations as in \cite{AndrenOdd}.

Let $X_{ij}$ be the indicator random variable
for the event that symbol $i$ occurs at least $\beta n$
times in row $j$ of $A$ and set
$$X = \sum_{1 \leq i,j \leq n} X_{ij}.$$
Similarly, let $Y_{ij}$ be the indicator random
variable for the event that symbol $i$ occurs at least
$\beta n$ times in column $j$ of $A$
and set $$Y = \sum_{1 \leq i,j \leq n} Y_{ij}.$$
Then we have
\begin{align}
\label{eq:Xij}
	\mathbb{P}[X > 0] & \leq  \mathbb{E}[X] 
	\leq n^2 \frac{\binom{n}{ \lceil \beta n \rceil} 
	\binom{n-1}{m -1}^{\lceil \beta n \rceil} \binom{n}{m}^{n^2 - 
	\lceil \beta n \rceil} }
	{\binom{n}{m}^{n^2}}
	\leq n^2 \frac{(n)_{\lceil \beta n \rceil}}{(\lceil \beta n \rceil)!} 
	\rho_2^{\lceil \beta n \rceil}
\end{align}
	where $(n)_{k}$ is the usual falling factorial. By applying Stirling's 
	formula, we see that 
	the right hand side of \eqref{eq:Xij} 
	tends to $0$ as $n \to \infty$, provided that 
	$\rho_2 < \frac{\beta}{e}$, where $e$ is the base of the natural logarithm.
	Proceeding similarly, if $\rho_2 < \frac{\beta}{e}$, then
	$\mathbb{P}[Y > 0] \to 0 \text{\quad as \, $n \to \infty$}$.
	Thus it follows that if $\rho_2 < \frac{\beta}{e}$, then
	the probability that $A$ is a $(\beta n,\beta n, \beta n)$-array 
	tends to $1$ as $n \to \infty$.
	
	Using calculations as above, it is straightforward to verify that if
	$\rho_1 \leq \frac{\alpha}{e}$, then with probability tending to $1$
	as $n \to \infty$, $P$ is $\alpha$-dense.
	
	Hence, by Theorem \ref{th:main}, the probability that
	there is a completion of $P$ that avoids
	$A$ tends to $1$ as $n \to \infty$. This concludes the proof
	of Theorem \ref{cor:random}.	
\bigskip

\noindent
{\bf Remark.} Note that the proof of Corollary \ref{cor:random} is valid if we take $P$ to be a random PLS and $A$ to be a given 
(deterministic) $(\beta n, \beta n, \beta n)$-array which the completion of $P$ should avoid; or, if we take $P$ to be a given 
$\alpha$-dense PLS and $A$ a random array. Furthermore, the proof of Corollary \ref{cor:random} is valid if 
$\rho_1 < \frac{\alpha}{e}$ and $\rho_2 < \frac{\beta}{e}$.  Thus if we can get better bounds on $\alpha$ and $\beta$ for which 
Theorem \ref{th:main} holds, then we also get a better bound on $\rho_1$
and $\rho_2$.

\section{Concluding Remarks}

We have proved that there are constants $\alpha$ and $\beta$ such that every $\alpha$-dense PLS can be completed to a Latin square $L$
that avoids a given $(\beta n, \beta n, \beta n)$-array, provided that the PLS avoids the array. Let us now briefly indicate what the best
possible values of $\alpha$ and $\beta$ might be.

In \cite{DaykinHaggkvist} it is conjectured that if $\alpha \leq \frac{1}{4}$, then any $\alpha$-dense PLS is completable; and in \cite{Haggkvist}
it is conjectured that if $\beta \leq \frac{1}{3}$, then any $(\beta n, \beta n, \beta n)$-array is avoidable. In \cite{Wanless}, for any $\gamma > 0$, 
examples of $(\frac{1}{4} + \gamma)$-dense partial Latin squares that are not completable are given;
 looking from the perspective of avoiding arrays, 
an example by Pebody shows for any $\gamma > 0$,  there are unavoidable 
$(\beta n, \beta n, \beta n)$-arrays with $\beta \geq 1/3 + \gamma$ (see e.g. \cite{CutlerOhman}).

We say that a point $(\alpha, \beta)$ is {\em feasible} if for every pair $(P,A)$, where $P$ is an $n \times n$ $\alpha$-dense
PLS and $A$ an $n \times n$ $(\beta n, \beta n, \beta n)$-array such that no entry of $P$ occurs in the corresponding cell of $A$, it is possible 
to complete $P$ into a Latin square that avoids $A$. A point which is not feasible is {\em infeasible}.
So the above examples show that the points $(0, 1/4+\gamma)$ and $(1/3+\gamma,0)$ are 
infeasible. Hence, the points outside the lines $(1/3,t)$ and $(t,1/4)$  are infeasible.

Using a combination of the mentioned constructions we can generate arbitrarily large examples of $\alpha$-dense partial Latin squares 
which  can not be completed to avoid a given $(\beta, \beta, \beta)$-array, provided that $\alpha + \beta = \frac{1}{3} + \gamma$, as follows:

For simplicity, assume that $n = 3 r +2$. Let $A$ be an  $(r+1) \times (r+1)$ array in which each cell contains the set 
$\{1,\dots, r+1\}$, let $B$ be an $(r+1) \times (r+1)$ array in which each entry is $\{r+2, \dots, 2r+2\}$, and let $C$ be an $r \times r$
arry in which each cell contains the set $\{2r+2, \dots, 3r+2\}$. Define $E_1$ to be the $n \times n$ array containing $A$ in the
upper left $(r+1) \times (r+1)$ corner, $B$ in the intersection of rows $r+2,\dots, 2r+2$ and columns $r+2,\dots, 2r+2$, and
$C$ in the lower right $r \times r$ corner.

\begin{equation*}
	E_1= \begin{array}{|c|c|cc|}
		\hline
		A& & & \\
		\hline
		 &B& &\\
		\hline
		 & &C&\\
		\hline
	\end{array}
\end{equation*}

The array $E_1$ is an unavoidable $(\beta n,\beta n,\beta n)$-array for, asymptotically,  $\beta=\frac{1}{3}$, see e.g. \cite{CutlerOhman}.

\begin{enumerate}

	\item We define three sets $S_1,S_2,S_3$ by setting 
	$$S_1=\{r+2\} \cup \{2r+3,\ldots, 3r+2\},
	S_2=\{1,\ldots,r+1\},
	S_3=\{r+3,\ldots, 2r+2\}.$$

	\item Following \cite{Wanless},  
	for each set $S_i$ we construct,
	an $|S_i| \times |S_i|$ single entry array $L_i$
	with symbols from $S_i$ such that each symbol occurs 
	precisely once in each row
	and column, and with the property that the cells of $L_i$ is the union of
	$|S_i|$ disjoint $S_i$-transversals $T_{i,j}$, $1 \leq j \leq |S_i|$,
	where an {\em $S_i$-transversal} is a 
	generalized diagonal in $L_i$ where each symbol 
	in $S_i$ occurs exactly once. For convenience, define 
	$T_{3, r+1}= \emptyset$.
	
	We now define an $n\times n$ PLS $E_2$ with $L_1$ in the 
	position held by $A$ in $E_1$, 
	$L_2$ in the position held by $B$ in $E_1$, and $L_3$ in the 
	position held by $C$ in $E_1$.

	\item Next, for each integer $t$ satisfying
	$1 \leq t \leq r+1$, define an $n \times n$ 
	array $E_{1t}$, from $E_1$ 
	by  setting $E_{1t}(p,c) = \emptyset$ for each position 
	$(p,c)$ of $E_1$ which corresponds to a nonempty cell $(p,c)_{E_2}$
	of $E_2$ such that $(p,c)_{E_2} \in 	\cup_{i}\cup_{j=1}^{t} T_{i,j}$. 
	We retain the content of any other cell of $E_1$.
	
	\item  We now define a PLS
	$E^1_{t}$ from $E_2$ by retaining the entry of each cell in
	$\cup_{i}\cup_{j=1}^{t} T_{i,j}$, and removing the entry of
	each cell in $E_2$ which does not belong to this set.

	\item  It follows that $E^1_t$ is a $\frac{t}{n}$-dense PLS, and $E_{1t}$ 
	is a $(\beta n-t,\beta n-t,\beta n-t)$-array.

\end{enumerate}

Now, the PLS $E^1_{t}$ cannot be completed to a Latin square 
which avoids $E_{1t}$; this follows from the fact that each cell in $E^1_{t}$
contains a symbol which does not occur in the corresponding
cell of $E_1$, and outside the support of $E^1_t$ (i.e. the non-empty cells of  $E^1_t$),
the array $E_{1t}$ agrees with $E_1$, so any Latin square which
is a completion of $E^1_t$ that avoids $E_{1t}$, would also avoid $E_1$.

Consider a line $\ell$ in the $\alpha \beta$-plane from $(1/3,0)$ to $(0,1/3)$.
The pairs $(E_{1t},E^1_t)$ yields that each point outside the region bounded by 
$\ell$ and the $\alpha$- and $\beta$-axis is infeasible. In fact, combined with the examples by Wanless,
we know that the set of feasible points is a subsets of region bounded by $\ell$, the line $(1/4,t)$ and the 
$\alpha$- and $\beta$-axis.

It would be interesting to obtain more information on the structure of set of feasible points, but we expect that 
other methods than those used in this paper will be needed for this.
Specifically, we would like to pose the following:
\begin{problem}
	Is  the set of feasible points $(\alpha, \beta)$ a convex set?
\end{problem}
Both of the conjectured boundary points $(0, 1/4)$ and $(1/3,0)$  are also boundary points for certain linear programming 
relaxations of the completion and avoidance problems \cite{Haggkvist2}. So, it might be possible to use a relaxation of the 
combined problem to provide a convex domain which gives a tighter bound for the set of feasible points than that given 
by our construction.

Further,  given that the constructions which give our bounds for the set of feasible points are highly structured and that our proof for 
Corollary \ref{cor:random} relies on our main result Theorem \ref{th:main}, it  is not unreasonable to expect that the best possible parameters in 
Corollary \ref{cor:random} are larger than those which even an optimal version of Theorem \ref{th:main} would give.  Here it would 
be interesting both to see if Corollary \ref{cor:random} can be improved and if some upper bounds on the possible values of 
$\rho_1$ and $\rho_2$ can be proven.



\begin{thebibliography}{99}
\addcontentsline{toc}{chapter}{Bibliography}

\bibitem{AdamsBryantBuchanan}
Peter Adams, Darryn Bryant and Melinda Buchanan, 
Completing partial Latin squares with two filled rows and two filled columns,
{\em Electronic Journal of Combinatorics} 15(1), R56, 26pp (2008).

\bibitem{AndersenHilton}
L. D. Andersen, A. J. W. Hilton,
\emph{Thank Evans!,
Proc. London Math. Soc.} 47 (1983), pp. 507--522.


\bibitem{AndrenEven}
L. J. Andr\'en,
\emph{On Latin squares and avoidable arrays},
Doctoral thesis, Ume\aa\enskip University, 2010.

\bibitem{AndrenOdd}
L. J. Andr\'en, C. J. Casselgren, L.-D. \"Ohman,
Avoiding arrays of odd order by Latin squares,
{\em Combinatorics, Probability and Computing} 22 (2013), 184--212.



\bibitem{Asratian} 
A. S. Asratian, T. M. J. Denley, R. H\"aggkvist,
\emph{Bipartite graphs and their applications},
Cambridge University Pres, Cambridge, 1998.

\bibitem{Bartlett}
P. Bartlett,
Completing $\epsilon$-dense partial Latin squares,
{\em Journal of Combinatorial Designs}
21 (2013), 447-–463.

\bibitem{Bregman}
L. M. Bregman,
Certain properties of nonnegative matrices and their permanents,
{\em Dokl. Akad. Nauk SSSR} 211 (1973), 27--30.

\bibitem{Casselgren}
C. J. Casselgren,
On avoiding some families of arrays,
{\em Discrete Mathematics}
312 (2012), 963--972. 

\bibitem{CasselgrenHaggkvist}
C. J. Casselgren, R H\"aggkvist,
Completing partial Latin squares with one filled row, column and symbol,
{\em Discrete Mathematics}
313 (2013), 1011--1017.


\bibitem{Cavenagh}
N. Cavenagh,
Avoidable partial latin squares of order 4m+1,
\emph{Ars Combinatoria} 95 (2010), pp. 257--275.

\bibitem{ChetwyndHaggkvist}
A. G. Chetwynd, R. H\"aggkvist,
Completing partial $n \times n$ Latin squares where each row, column
and symbol is used at most $cn$ times, 
Research report, Dept. of Mathematics,
Stockholm University, 1984.


\bibitem{ChetwyndChess}
A. G. Chetwynd, S. J. Rhodes,
Chessboard squares,
\emph{Discrete Mathematics} 141 (1995), pp. 47--59.

\bibitem{ChetwyndRhodes1}
A. G. Chetwynd, S. J. Rhodes,
Avoiding partial Latin squares and intricacy
\emph{Discrete Mathematics} 177 (1997), pp. 17--32.

\bibitem{ChetwyndRhodes2}
A. G. Chetwynd, S. J. Rhodes,
Avoiding multiple entry arrays,
\emph{Journal of Graph Theory} 25 (1997), pp. 257--266.




\bibitem{Colbourn}
C. J. Colbourn,
The complexity of completing partial Latin squares,
\emph{Discrete Applied Mathematics} 8 (1984), pp. 25--30.


\bibitem{CutlerOhman}
J. Cutler, L.-D. \"Ohman,
Latin squares with forbidden entries,
\emph{Electronic Journal of Combinatorics} 13 (2006), 9 pp. (electronic).

\bibitem{DaykinHaggkvist}
D. E. Daykin, R. H\"aggkvist,
Completion of sparse partial Latin squares,
{\em Graph Theory and Combinatorics: Proceedings of the Cambridge Conference
in Honor of Paul Erd\H os} (1984), 127--132.


\bibitem{DenleyKuhl}
T. Denley, J. Kuhl,
Constrained completion of partial Latin squres,
{\em Discrete Mathematics} 312 (2012), 1251--1256.

\bibitem{Evans}
T. Evans, Embedding incomplete latin squares,
\emph{American Mathematical Monthly} 67 (1960),
958--961.

\bibitem{Galvin}
F. Galvin,
{\em The list-chromatic index of a bipartite multigraph}
Journal of Combinatorial Theory, Series B 63 (1995), 153--158.

\bibitem{Gustavsson}
T. Gustavsson,
{\em Decompositions of large graphs and digraphs with high minimum degree},
Doctoral thesis, Stockholm University, 1991.

\bibitem{Haggkvist}
R. H\"aggkvist,
A note on Latin squares with restricted support,
\emph{Discrete Mathematics} 75 (1989), pp. 253--254.

\bibitem{Haggkvist2}
R. H\"aggkvist, \emph{Personal communication}.

\bibitem{KuhlSchroeder}
J. S Kuhl, M. Schroeder,
Completing Partial Latin Squares with One Nonempty Row, Column, and Symbol,
{\em Electronic Journal of Combinatorics} Volume 23, Issue 2 (2016).


\bibitem{MarkstromOhman}
K. Markstr\"om, L.-D. \"Ohman,
Unavoidable arrays,
\emph{Contributions to Discrete Mathematics} (2009), 90--106.


\bibitem{Ryser}
H. J. Ryser, 
A combinatorial theorem with an application to Latin squares,
\emph{Proc. Amer. Math. Soc.} 2 (1951), 550--552.

\bibitem{Smetaniuk}
B. Smetaniuk,
A new construction for Latin squares I. Proof of the Evans conjecture,
\emph{Ars Combinatoria} 11 (1981), 155--172.

\bibitem{Wanless}
I. Wanless,
A generalization of transversals for latin squares,
{\em Electronic Journal of Combinatorics} 2 (2002).


\bibitem{Ohman}
 L.-D. \"Ohman,
Partial latin squares are avoidable,
\emph{Annals of Combinatorics} 15 (2011), 485--497.

\bibitem{Ohman2}
L.-D. \"Ohman,
Latin squares with prescriptions and restrictions,
{\em Australasian Journal of Combinatorics} 51 (2011), 77--87.




\end{thebibliography}
\end{document}